\documentclass[final,leqno,onefignum,onetabnum]{siamltex1213}
\newcommand{\mytitle}{Accelerated Portfolio Optimization with Conditional Value-at-Risk Constraints using a Cutting-Plane Method}
\usepackage{amsmath}
\usepackage{amssymb}
\usepackage{multirow}

\newcommand{\R}{{\mathbb R}}
\newcommand{\N}{{\mathbb N}}

\newcommand{\T}{\mathrm{T}}
\usepackage{listings}
\usepackage{color}

\definecolor{dkgreen}{rgb}{0,0.6,0}
\definecolor{gray}{rgb}{0.5,0.5,0.5}
\definecolor{mauve}{rgb}{0.58,0,0.82}

\begin{document}
    \title{\mytitle\thanks{The US patent \cite{OptPatent} uses methods from this article.}}
    \author{Georg Hofmann\thanks{The Research for this paper is supported by Validus Research Inc., Waterloo Ontario, Canada.}}
    \maketitle

\begin{abstract}
  Financial portfolios are often optimized for maximum profit while subject to a constraint formulated in terms of the Conditional Value-at-Risk (CVaR). This amounts to solving a linear problem. However, in its original formulation this linear problem has a very large number of linear constraints, too many to be enforced in practice. In the literature this is addressed by a reformulation of the problem using so-called dummy variables. This reduces the large number of constraints in the original linear problem at the cost of increasing the number of variables. In the context of reinsurance portfolio optimization we observe that the increase in variable count can lead to situations where solving the reformulated problem takes a long time. Therefore we suggest a different approach. We solve the original linear problem with cutting-plane method: The proposed algorithm starts with the solution of a relaxed problem and then iteratively adds cuts until the solution is approximated within a preset threshold. This is a new approach. For a reinsurance case study we show that a significant reduction of necessary computer resources can be achieved.
\end{abstract}



\section{Introduction}

In the industry, risk management of financial portfolios is 
commonly 
ap\-proached with a Monte Carlo simulation. The model typically consists of a $J{\times}n$-matrix $Y$. Its $J$ rows represent scenarios, i.e. outcomes of a simulation that are considered to be equally probable. Its $n$ columns represent different instruments that make up the portfolio. The entries of $Y$ represent the value of an instrument for a particular outcome. So the column means of $Y$ yield the expected value of the instruments. The row sums of $Y$ yield the portfolio value for each scenario. We call this the \emph{outcome vector}. In the context of reinsurance portfolios the value of instruments can become negative due to catastrophic loss simulated in the scenarios. However the expected instrument value is typically positive due to premium collected. We refer to the matrix $Y$ as the \emph{scenario matrix}.

In this article we represent the risk of a portfolio by the Conditional Value-at-Risk (CVaR), also called the Tail Value-at-Risk (TVaR).\footnote{For a a valuable discussion of risk metrics commonly used in this context, see the introduction of \cite{Krokhmal02portfoliooptimization}. Of particular interest is the comparison with the most common risk metric, the Value-at-Risk (VaR) which is simply a percentile of the loss distribution.} 
In the model above the CVaR at a return period $\rho$ can be calculated as follows, provided that the number of scenarios $J$ is a multiple of the return period $\rho$: Let $y$ be the row sums of $Y$. Let $y'$ be a reordering of $y$ in a way that the components $y'$ are increasing. Define the vector $r$ by setting
  \begin{align*}
    r_j=
    \begin{cases}
      -\frac{\rho}{J}&\text{for~} i=1,~2,\dots,~\frac{J}{\rho}\\
      0&\text{for~} i=\frac{J}{\rho}+1,~\frac{J}{\rho}+2,\dots,~J.
    \end{cases}
  \end{align*}
for every $j=1,~2,~,\dots,~J$.
Then the desired CVaR is given by the matrix product $r^\T y'$. An equivalent way of defining the CVaR of $y$ is the following:
\begin{align}\label{eq:muR}
\mu_r(y)=\max_{\pi\in S_J}r^\T P_\pi y,
\end{align}
where $S_J$ is the set of all permutations on the set $\{1,~2,~,\dots.,~J\}$ and $P_\pi$ is the permutation matrix associated with the permutation $\pi$. We refer to $r$ as the \emph{risk vector} associated with the \emph{risk metric} $\mu_r$. 

In order to model change in the portfolio composition an $n$-dimensional vector $x$ is used that is subject to the constraints 
\begin{align}\label{eq:posConstr1}
\underline x\le x\le\overline x,
\end{align} 
where $\underline x$ and $\overline x$ are both $n$-dimensional vectors. The inequality~(\ref{eq:posConstr1}) is to be read component by component. In other words, if $x_i$ denotes the $i$th component of $x$ then we enforce the following inequality for every $i=1,~2,~,\dots,~n$:
\begin{align*}
\underline x_i\le x_i\le\overline x_i.
\end{align*} 
We refer to $x$ as a \emph{position vector} and to (\ref{eq:posConstr1}) as the \emph{position constraint}. For a given position vector $x$ the altered scenario matrix is obtained by multiplying the columns of $Y$ with the components of $x$. The altered outcome vector is simply the product $Yx$.

Suppose the risk of the altered portfolio to be bounded by a number $R$, in other words the inequality
\begin{align*}
\mu_r(Yx)\le R
\end{align*}
is to be enforced. An equivalent way of formulating this constraint is
\begin{align*}
 rP_\pi Yx\le R\text{~~~for every permutation $\pi\in S_J$}.
\end{align*}
This is a system of $J!$ linear inequalities. The number of inequalities could be reduced to account for the fact that $r$ contains zeros, but it is generally still very large. In practice this prevents the problem from being passed to a standard algorithm in its original formulation.

Recently a practical solution has been presented in \cite{Rockafellar00optimizationof} and refined in \cite{Uryasev00Conditional} and \cite{Krokhmal02portfoliooptimization}. It is based on a reformulation of the problem that increases the number of variables from $n$ to $n+J+1$ but decreases the number of constraints to $2J+2n+1$. In many cases this makes the problem accessible to standard algorithms for linear problems. It allows to process 'over one hundred instruments and one thousand scenarios'. (See  \cite{Krokhmal02portfoliooptimization}, Introduction, page 2.)

 However, as the number of scenarios $J$ increases, these algorithms may still take a long time to complete or fail due to a overly high demand for computer resources. In the reinsurance industry a simulation with 1 million scenarios is not an exception, neither is the usage of 10 thousand instruments. 

We propose a different approach to solving the linear problem at hand. We address the original problem using a cutting-plan method. The iterative algorithm starts with only the position constraints enforced. Then, step by step, relevant constraints, so-called cuts, are enforced. The algorithm terminates once the set of constraints is large enough to force the observed risk $R^*$ to be sufficiently close to $R$. To be more precise, a \emph{risk error tolerance} $\delta$ can be specified. The termination condition is then given by:
\begin{align*}
  \left|\frac{R^*-R}{R}\right|\le \delta
\end{align*}

In a case study we mimic typical reinsurance portfolio scenarios of different sizes. We provide the number of iterations when our algorithm is applied and also give a comparison of run times when the algorithm proposed by \cite{Rockafellar00optimizationof} and our algorithm are applied. We were able to run our algorithm for one million scenarios and 10 thousand instruments, and this data size does not form an upper bound.
  
While this case study is motivated from practice in the reinsurance industry, the algorithm we propose is not at all limited to this industry. It is applicable to optimization of any financial portfolio. In Section~\ref{sec:extension} we list some interesting linear constraints that an extension of the proposed algorithm can handle. Applications beyond the financial portfolios could be found is in other areas of risk management.

\section{The Linear Problem}

Throughout the remainder of the article assume that the number of instruments $n$ and the number of scenarios $J$ are given. Let $x$ be a position vector subject to the position constraint 
\begin{align}\label{eq:posConstr}
\underline x\le x\le\overline x.
\end{align} 
Fix a scenario matrix $Y$. Let $r$ be a $J$-dimensional vector with non-positive components. The example of a TVaR risk vector from the introduction is instructive. But other choices are possible, in particular a weighted combination of different TVaR risk vectors is admissible. In addition to constraint (\ref{eq:posConstr}) we enforce
\begin{align}\label{eq:riskConstr2}
 rP_\pi Yx\le R\text{~~~for every permutation $\pi\in S_J$}.
\end{align}
Let $p$ be an $n$-dimensional vector. It is helpful to think of $p$ has the vector column sums of $Y$. This way, the $p^\T x$ represents the expected altered portfolio value. But other choices of $p$ are possible.

The linear problem at hand is to

\begin{align}\label{eq:orrigProb}
\text{Maximize $p^\T x$ subject to the constraints (\ref{eq:posConstr}) and (\ref{eq:riskConstr2}).}
\end{align}

Throughout the remainder of the article we will refer to this as the \emph{original linear problem}. If there is at least one vector $x$ that satisfies the constraints of the problem, then the problem has a solution since $x$ is constrained to a compact domain. Typically the $n$ dimensional vector of ones lies within the constraints, as it represents the unaltered portfolio, which usually satisfies the risk constraints.

\section{The Algorithm}

In practice the number of constraints in (\ref{eq:riskConstr2}) is too large for all of them to be enforced. We propose the following algorithm to add only constraints that are relevant for the solution. This approach is similar to the cutting-plane method. 

The following inputs are required. We use the terminology and the notation of previous sections. The following are the input parameters:
\begin{center}
\begin{tabular}{ll}
$Y$&Scenario matrix\\
$p$&Profit vector\\
$\overline x$&Upper position constraint vector\\
$\underline x$&Lower position constraint vector\\
$R$&Target risk\\
$\delta$&Risk error tolerance
\end{tabular}
\end{center}
The following steps describe the algorithm.
\begin{enumerate}
\item
  \emph{Formulate the relaxed problem:} Initiate the set of constraints $C$ as the constraints defined in (\ref{eq:posConstr}).
\item
  \emph{Solve the current problem:} Maximize $p^\T x$ subject to the constraints in $C$. Denote the solution by $x^*$.
\item 
  \emph{If the achieved risk $R^*$ is close enough to the target risk $R$, go to Step~6:} Set 
  \begin{align*}
  R^* = \mu_r(Yx^*).
  \end{align*}
  If $|R^*-R|\le\delta R$ then jump to Step~6.
\item
  \emph{Add a constraint to the problem:} Let $\pi$ be a permutation such that $P_\pi Yx^*$ is in increasing order. Add the constraint
  \begin{align*}
   rP_\pi Yx\le R
  \end{align*}
  to the set $C$.
\item     
  \emph{Return to Step~2.}
\item 
  \emph{Optional verification of the solution to the relaxed problem:} By solving the dual problem to the one formulated in Step~2, it can be verified that an optimal solution has been obtained.  
\item 
\emph{Algorithm output:} The output of the algorithm is the solution $x^*$, the achieved risk $R^*$ and the obtained profit $s=p^\T x^*$.   
\end{enumerate}

This is an iterative algorithm and its output is a numerical approximation of the theoretical solution. In the next section we describe the nature of this approximation and how its error can be quantified. In the section following that we provide a case study that includes the number of iterations observed for actual data. In reality the approximation error is comparable to rounding errors that are typical for any numerical solution.

\section{Convergence}

In this section we give a precise description of how the solution achieved form the iterative algorithm approximates the solution of the original linear problem~(\ref{eq:orrigProb}). Let $D$ be the set of all numbers $R'$ for which the following set is not empty:
\begin{align*}
\{x'\in[\underline x,\overline x]:\mu_r(Yx')\le R'\}
\end{align*}
Note that the target risk $R$ should be in this set, otherwise constraints~(\ref{eq:posConstr}) and (\ref{eq:riskConstr2}) can't be satisfied.
For every $R'\in D$ set
\begin{align}
  f(R')=&\max_{
  \begin{array}{c}
    \underline x\le x'\le\overline x\\
    \mu_r(Yx')\le R'
  \end{array}
  } p^\T x'.
\end{align}
This is well-defined, since we are taking the maximum of a continuous function over a compact set. So, this assignment defines a function from $D$ to $\R$. The function $f$ is called the \emph{efficient frontier}. Note that $f(R)$ is precisely the maximum profit that is to be determined in the original problem~(\ref{eq:orrigProb}).


\begin{figure*}[t]
      \centering
      \vspace{-.5cm}    
      \par\includegraphics[width=20cm]{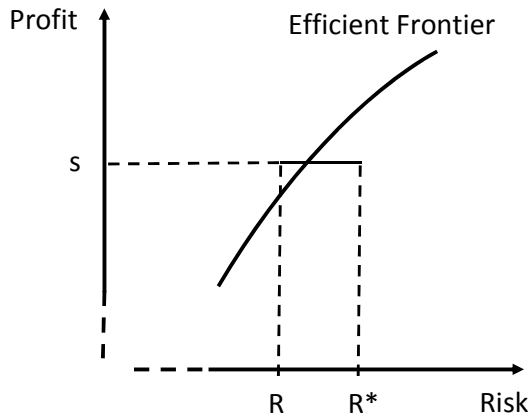}  
      \par\vspace{-8cm}   
      \caption{The line segment from $(R,s)$ to $(R^*,s)$ intersects the efficient frontier. This guarantees that the algorithm approximates a point on the frontier.}
      \label{fig:efficientFrontier}
\end{figure*}
Now consider the algorithm again with the target risk $R$. Assume that a solution to the original problem exists. Then solutions exists to any of the relaxations of the problem occurring in the algorithm. Suppose the algorithm terminates with the achieved risk $R^*$ and the obtained profit $s$. Since these values resulted from a relaxed problem we have $s\ge f(R)$. On the other hand, due to the definition of $f$, we have $s\le f(R^*)$.

This means that there is a number $R'$ with 
$R\le R'\le R^*$ such that $f(R')=s$. In other words, there is a frontier point between the points $(R,s)$ and $(R^*,s)$ in the efficient frontier diagram. See Figure~\ref{fig:efficientFrontier}.
This follows from the Intermediate Value Theorem. A proof that $f$ is actually continuous is provided in Appendix~\ref{sec:Proof}.

Since the distance between $R$ and $R^*$ can be controlled by specifying $\delta$, the distance between the output point $(R^*,s)$ and a point on the frontier is also controlled by $\delta$. 


\section{A Case Study}

In this section, the application of the two methods is compared:
\begin{enumerate}
\item[(A)] The method proposed in \cite{Rockafellar00optimizationof}, which relies on a reformulation of the underlying linear problem.
\item[(B)] The method proposed in this article, which uses a cutting-plane approach to address the underlying linear problem.
\end{enumerate}
A few combinations of values for $n$ and $J$ are selected and an example of $Y$ is generated to allow a comparison on actual data. The scenario matrix $Y$ created randomly with a typical reinsurance portfolio in mind. However it is important to emphasize that it is not based on any true data and that it should not be used to represent in any way a real portfolio.

The matrix $Y$ is created in the following way. An $f{\times}n$ matrix $L$ is created where $f=100$ in this case study. Its entries are randomly generated between 0 and 1 according to a uniform distribution. A $J{\times}f$ matrix $F$ is created. Each entry is randomly generated.  The underlying distribution is derived from a log normal by an affine transformation such that defined as followed: If $N$ is distributed according to a standard normal distribution, then the entries of $F$ are sampled from the distribution of the transformed variable $2-e^N$. Its support is $(-\infty, 2]$ and its expected value is $2-\sqrt e\approx 0.351$. The matrix $Y$ is finally computed as the product $FL$. In this way each instrument is a random linear combination of $f$ factors. This introduces a correlation between instruments that we consider typical for a reinsurance portfolio.

In Appendix~\ref{sec:rCode}, code in the language R is provided to show how the scenario matrix $Y$ is created and how the algorithm is implemented. 
\begin{table}
\centering
\begin{tabular}{|r|r||r|r|r|r|r|r|}
  \hline
  \multicolumn{1}{|c|}{\multirow{2}{*}{$J$}} &
  \multicolumn{1}{|c||}{\multirow{2}{*}{$n$}} &
  \multicolumn{2}{|c|}{Iterations} &
  \multicolumn{2}{|c|}{Variables} &
  \multicolumn{2}{|c|}{Constraints}\\
  \cline{3-8}
  & &
  \multicolumn{1}{|c|}{(A)}&
  \multicolumn{1}{|c|}{(B)}&
  \multicolumn{1}{|c|}{(A)}&
  \multicolumn{1}{|c|}{(B)}&
  \multicolumn{1}{|c|}{(A)}&
  \multicolumn{1}{|c|}{(B)}\\
  \hline
     1,000 &   100 & 1 &    4 &     1,101 &   100 &     2,201 &   204\\
    10,000 &   200 & 1 &   14 &    10,201 &   200 &    20,401 &   414\\
   100,000 &   500 & 1 &   58 &   100,501 &   500 &   201,001 & 1,058\\
 1,000,000 & 1,000 & 1 &  223 & 1,001,001 & 1,000 & 2,002,001 & 2,223 
   \\\hline
\end{tabular}
\caption{Comparison of Approach~(A) and (B)}  
\label{tab:compApproach}
\end{table}
For  each of the two methods Table~\ref{tab:compApproach} provides information about the number constraints and variables in the linear problem to be solved. Since method (B) is iterative, several linear problems have to be solved in the course of the algorithm, as the number of constraints, we report the number of constraints in the last iteration. 
Table~\ref{tab:accel} provides a comparison in run times. 
\begin{table}
\centering
\begin{tabular}{|r|r|r|r|r|}
\hline
  $J\backslash n$ & 100 & 200 & 500 & 1,000\\\hline
  1,000 & 13 & 13 & 11 & 8 \\
  2,000 & 22 & 26 & 19 & 15 \\
  5,000 & 43 & 48 & 43 & 41 \\
  10,000 & 96 & 90 & 115 & 97 \\
  \hline
\end{tabular}
\caption{Acceleration of run time when moving from Approach~(A) to Approach~(B).}  
\label{tab:accel}
\end{table}
The factor report is the run time of method~(A) divided by method~(B). For both tables the CVaR at a return period of 100 was used.

\section{Extensions of the Algorithm}\label{sec:extension}

As discussed earlier, the application of the algorithm is not limited to the reinsurance case investigated in the previous section. It can be used to optimize financial portfolios in general. Further applications in areas of risk management are likely. 

The algorithm can also be extended to handle further linear constraints. These include the constraints described in \cite{Krokhmal02portfoliooptimization}: In that article they are referred to as 
\begin{enumerate}
\item Transaction Cost Balance Constraints (\cite{Krokhmal02portfoliooptimization}~7.3)
\item Value Constraints (\cite{Krokhmal02portfoliooptimization}~7.4)
\item Liquidity Constraints (\cite{Krokhmal02portfoliooptimization}~7.5)
\end{enumerate}
Note that the constraints referred to as {\it Bounds on Positions} in \cite{Krokhmal02portfoliooptimization}~7.5, are effectively the constraints we cover with (\ref{eq:posConstr}).

In practice it is often desirable to calculate segments of the efficient frontier describing the trade-off between risk and profit. The proposed algorithm can be used to accomplish this by stepping through a range of risk values. For each risk value the optimal profit value is calculated. This method can provide a desirable range of points on the frontier.

\section{Conclusion}

We present a novel approach to solving an existing problem in practice. In a case study that is geared towards reinsurance business, we show that a significant gain in performance and a reduction in computer resources can be achieved using our method. We prove that the proposed algorithm approximates an efficient frontier point within a customizable accuracy. With the optional validation step in the algorithm, it can be guaranteed that an optimal solution has been captured.

\appendix

\section{Proposition and Proof}\label{sec:Proof}

Let $X$ be a non-empty compact convex subset of $\R^n$. Let 
\begin{align*}
\mu:~X\to\R
\end{align*}
be a continuous, convex function. Note that the image $\mu(X)$ is compact so we can set 
\begin{align*}
d=\min(\mu(X)).
\end{align*}
Let 
\begin{align*}
p:\R^n\to\R
\end{align*}
be a concave continuous function and define the function 
\begin{align}
  f:[d,\infty)\to\R,~f(R)=&\max_{x\in\mu^{-1}\big((-\infty, R]\big)}p(x)\\
    =& \max_{
    \substack{
      x\in X\\ 
      \mu(x)\le R
    }
    }p(x)
\end{align}
It is well-defined, since $\mu^{-1}\big((-\infty, R]\big)$ is compact and $p$ is continuous. Note that $f$ is an increasing function.
\begin{proposition}
The function $f$ is continuous.
\end{proposition}
\begin{proof}
First we prove that $f$ is concave. That implies that $f$ is continuous on any open interval. Once this is established, the only way that continuity of $f$ can fail is that it jumps at $d$. In a second step we will show that this is not possible.

For the proof of concavity let $R_1$ and $R_2\in D$. Now let $x_1$ and $x_2\in X$ such that
\begin{align}\label{eq:xDef}
x_i\in \mu^{-1}\big((-\infty, R]\big)\text{~~~and~~~} f(R_i)=p(x_i)
\end{align}
for each $i=1,2$. Now let $t\in[0,1]$. Then
\begin{align*}
 \mu(tx_1 + (1-t)x_2)
 ~\le~& t\mu(x_1)+(1-t)\mu(x_2)&\text{since $\mu$ is convex}\\
 ~\le~& tR_1+(1-t)R_2&\text{according to (\ref{eq:xDef})},
\end{align*}
in other words
\begin{align}\label{eq:concave1}
 tx_1 + (1-t)x_2
 \in \mu^{-1}\Big(\big(-\infty, tR_1+(1-t)R_2\big)\Big).
\end{align}
In turn, this implies
\begin{align*}
tf(R_1)+(1-t)f(R_2)
~=~& tp(x_1) + (1-t)p(x_2)&\text{according to (\ref{eq:xDef})}\\
\le~& p(tx_1 + (1-t)x_2)&\text{since $p$ is concave}\\
\le~&\max_{x\in\mu^{-1}\big((-\infty,tR_1+(1-t)R_2]\big)}p(x)&\text{due to (\ref{eq:concave1})}\\
=~&f(tR_1+(1-t)R_2).
\end{align*}
This proves that $f$ is concave.

Now let $(R_i)$ be a sequence in $(d,\infty)$ with $\lim_{i\to\infty}R_i=d$. We will prove \begin{align}\label{eq:cont}
\lim_{i\to\infty}f(R_i)= f(d)
\end{align} in order to see that $f$ is continuous at $d$. Without loss of generality we may assume that $(R_i)$ is decreasing. Since $f$ is an increasing function the sequence $\big(f(R_i)\big)$ is also decreasing. Since it is bounded from below by 
$\min p(X)$ it is convergent. By the definition of $f$ there is an $x_i\in X$ such that 
\begin{align*}
\mu(x_i)\leq R_i \text{~~~and~~~} p(x_i)=f(R_i)
\end{align*}
for every $i\in\N$. In this way we obtain a sequence $(x_i)$. Since $X$ is compact, there is a subsequence $(x_{i_j})$ that converges to an $x\in X$. We observe that
\begin{align*}
f(d)\le&\lim_{j\to\infty}f(R_{i_j})&\text{since $f$ is increasing}\\
=&
\lim_{j\to\infty}p(x_{i_j})
=p(x)\\
\le&
\max_{x'\in X,~ \mu(x')\le d}p(x')
=
f(d)\\
\end{align*} 
This implies $\lim_{j\to\infty}f(R_{i_j})=f(d)$ and thus (\ref{eq:cont}). 
\end{proof}

This proposition can be applied to the case 
\begin{align*}
\mu(x)=\mu_r(Yx)
\end{align*}
for the function $\mu_r$ defined in (\ref{eq:muR}). For that it should be noted that $\mu_r$ is continuous, since it is the maximum of a finite number of continuous functions.

\section{R Code}\label{sec:rCode}
    
The code below runs in R version 3.0.2 once the package {\tt lpSolve} is loaded. This package is based on {\tt lp\_solve} 5.5.
\lstset{language=R}
\begin{lstlisting}
# This code is part of the article "Accelerated Portfolio Optimization 
# with Conditional Value-at-Risk Constraints using a Cutting-Plane Method
library(lpSolve)

### Preset constants.
num.scenarios <- 1E4
num.instruments <- 1000

### Functions
GetRiskMetricVector <- function(return.period){
  rp.scenarios <- num.scenarios / return.period
  return(c(rep(-1 / rp.scenarios, rp.scenarios), 
    			 rep(0, num.scenarios - rp.scenarios)))
}

### Main Code
# The variables F, L, Y, p, r, delta are explained in the article.
f=100
F <- matrix(2 - rlnorm(n=num.scenarios * f), nrow=num.scenarios)
L <- matrix(runif(num.instruments * num.factors), nrow=f)
Y <- F %*% L
delta <- 1E-6
# As a risk metric we use the average of the CVaR 100 and the CVar 1000
r <- 0.5 * GetRiskMetricVector(100) + 0.5 * GetRiskMetricVector(1000)
A <- rbind(diag(nrow=num.instruments), diag(nrow=num.instruments))
b <- c(rep(1.5, num.instruments), rep(0.5, num.instruments))
p <- colSums(Y) / num.scenarios
constr.vec <- c(rep("<=", num.instruments), rep(">=", num.instruments))
# We set the risk constraint to the level of risk in the orriginal portfolio:
scenario.outcome <- rowSums(Y)
scenario.outcome <- scenario.outcome[order(scenario.outcome)]
risk.constraint <- r %*% scenario.outcome
repeat{
  lp.sol <- lp(direction="max", objective.in=p, const.mat=A, 
  						 const.dir=constr.vec, const.rhs=b)
  if(lp.sol$status>0)stop("Lp solve error: ", lp.sol$status)
  scenario.outcome <- Y %*% lp.sol$solution
  new.order <- order(scenario.outcome)
  inv.new.order <- invPerm(new.order)
  r.reordered <- r[inv.new.order]
  achieved.risk <- as.vector(r.reordered %*% scenario.outcome)
  if(achieved.risk - risk.constraint < delta) break 
  A <- rbind(A, as.matrix(t(r.reordered) %*% Y))
  b <- c(b, risk.constraint)
  constr.vec <- c(constr.vec, "<=")
}
cat("Risk constraint:", risk.constraint, "\r\n")
cat("Achieved risk:", achieved.risk, "\r\n")
cat("Profit:", lp.sol$objval, "\r\n")
\end{lstlisting}

\bibliographystyle{siam}
\bibliography{../../../Latex/references/references}

\end{document}